\documentclass[reqno,11pt]{amsart}
\usepackage{amsmath, amsfonts,amsthm,amssymb,amsbsy,upref,color,graphicx,amscd,hyperref,enumerate}
\usepackage[active]{srcltx}
\usepackage[latin1]{inputenc}
\usepackage{bbm}
\usepackage{algorithm}
\usepackage{algorithmicx}
\usepackage{algpseudocode}
\usepackage{subfig} 
\usepackage[a4paper,margin=2.96truecm]{geometry}
\usepackage{pgf}
\usepackage{xfrac}

\def\eps{{\varepsilon}}

\def\O{\Omega}
\def\R{\mathbb{R}}

\def\V{\mathcal{V}}
\newcommand{\be}{\begin{equation}}
\newcommand{\ee}{\end{equation}}
\newcommand{\bib}[4]{\bibitem{#1}{\sc#2: }{\it#3. }{#4.}}

\def \ba {\begin{array}}
\def \ea {\end{array}}
\def \dis {\displaystyle}

\numberwithin{equation}{section}
\theoremstyle{plain}
\newtheorem{theo}{Theorem}[section]
\newtheorem{lemm}[theo]{Lemma}

\theoremstyle{remark}
\newtheorem{rema}[theo]{Remark}

\newenvironment{ack}{{\bf Acknowledgements. }}

\title[Optimal potentials for problems with changing sign data]{Optimal potentials for problems\\ with changing sign data}

\author{Giuseppe Buttazzo, Faustino Maestre, Bozhidar Velichkov}

\begin{document}

\maketitle

\begin{abstract}
We consider optimal control problems where the state equation is an elliptic PDE of a Schr\"odinger type, governed by the Laplace operator $-\Delta$ with the addition of a potential $V$, and the control is the potential $V$ itself, that may vary in a suitable admissible class. In a previous paper (Ref. \cite{bgrv14}) an existence result was established under a monotonicity assumption on the cost functional, which occurs if the data do not change sign. In the present paper this sign assumption is removed and the existence of an optimal potential is still valid. Several numerical simulations, made by {\tt FreeFem++}, are shown.
\end{abstract}

\medskip\textbf{Keywords:} Schr\"odinger operators, optimal potentials, shape optimization, free boundary, capacitary measures, stochastic optimization

\medskip\textbf{2010 Mathematics Subject Classification:} 49J45, 49Q10, 35J10, 49A22, 35J25, 49B60

\section{Introduction and statement of the problem}\label{sintro}

In the present paper we consider optimization problems of the form
\be\label{minpb}
\min\bigg\{\int_D g(x)u(x)\,dx\ :\ -\Delta u+Vu=f,\ u\in H^1_0(D),\ V\in\V\bigg\}.
\ee
Here $D$ is a fixed bounded domain of $\R^d$, $f$ and $g$ are two given functions in $L^2(D)$, and the potential $V$ may vary in the admissible class $\V$ which is described below. Problem \eqref{minpb} is then an optimal control problem where $H^1_0(D)$ is the space of states, $\V$ is the set of admissible controls, $-\Delta u+Vu=f$ is the state equation, and $\int_D g(x)u(x)\,dx$ is the cost functional.

Problems of this form have been considered in \cite{bgrv14} under some assumptions on the admissible class $\V$. In particular, the admissible class $\V$ was taken of the form
$$\V=\left\{V:D\to[0,+\infty]\ :\ V\hbox{ Lebesgue measurable, }\int_D\Psi(V)\,dx\le1\right\}$$
with the function $\Psi$ satisfying some qualitative conditions. For instance, in order to approximate shape optimization problems with Dirichlet condition on the free boundary, the choice
$$\Psi(s)=e^{-\alpha s}$$
with $\alpha$ small, was proposed. More precisely, as $\alpha\to0$ the problems with the parameter $\alpha$ were shown to $\Gamma$-converge to the shape optimization problem with a volume constraint $|\O|\le1$ being $\O$ the shape variable. The existence of an optimal potential $V_{opt}$ was shown under the key assumption (see Theorem 4.1 of \cite{bgrv14}) to have a cost functional depending on the potential $V$ in a monotonically increasing way. This occurs, by the maximum principle, when $f\ge0$ and $g\le0$, and in this case the constraint is saturated, in the sense that $\int_D\Psi(V_{opt})\,dx=1$.

When the data $f$ and $g$ are allowed to change sign, the structure of the proof above is not valid any more and the question of the existence of an optimal potential was open. Similar questions arise for shape optimization problems, where again the monotonicity of the cost plays a crucial role.

The case of shape optimization problems with changing sign data was recently considered in \cite{buve17} where a new approach was proposed, allowing to obtain the existence of optimal shapes in a larger framework allowing general functions $f$ and $g$. We adopt here an approach similar to the one of \cite{buve17}, adapted to treat the case of potentials. Of course, when $f$ and $g$ may change sign, the constraint does not need to be saturated, in the sense that we may expect for an optimal potential $V_{opt}$ some situations in which $\int_D\Psi(V_{opt})\,dx<1$.

Problems of the kind considered here intervene in some variational problems with uncertainty, where the right-hand side $f$ is only known up to a probability $P$ on $L^2(D)$ (see for instance \cite{buve17} for the shape optimization framework). Other kinds of uncertainties can be treated by the so-called {\it worst case analysis}; in the case of shape optimization problems we refer for this topic to \cite{all14}, to \cite{bebuve} and to references therein.

We stress the fact that in our case the assumption that the cost function is linear with respect the state variable $u$ is crucial; otherwise simple examples show that an optimal shape or an optimal potential may not exist (see for instance \cite{bubu05}, \cite{bdm91} and \cite{bdm93}) and the optimal solution only exists in a relaxed sense in the space of capacitary measures, introduced in \cite{dmmo87}.

In Section \ref{sexist} we give the precise statement of the existence result and its proof. In Section \ref{snec} we provide some necessary conditions the optimal potentials have to fulfill. Finally, in Section \ref{snumer} we provide several numerical simulations that show the optimal potentials in some two dimensional cases.

\section{Existence of optimal potentials}\label{sexist}

In this section we consider the optimization problem \eqref{minpb} with
$$\V=\left\{V:D\to[0,+\infty]\ :\ V\hbox{ Lebesgue measurable, }\int_D\Psi(V)\,dx\le1\right\}.$$
On the function $\Psi:[0,+\infty]\to[0,+\infty]$ we assume that:
\begin{itemize}
\item[i)]$\Psi$ is strictly decreasing;
\item[ii)]there exist $p>1$ such that the function $s\mapsto\Psi^{-1}(s^p)$ is convex.
\end{itemize}
For instance the following functions:
\begin{enumerate}
\item $\Psi(s)=s^{-p}$, for any $p>0$,
\item $\Psi(s)=e^{-\alpha s}$, for any $\alpha>0$,
\end{enumerate}
satisfy the assumptions above. We always assume that the admissible class $\V$ is nonempty, that is $|D|\Psi(+\infty)\le1$.

It is known that the relaxed form of the optimization problem \eqref{minpb} involves capacitary measures, that is nonnegative Borel measures on $D$, possibly taking the value $+\infty$, that vanish on all sets of capacity zero. For all the details about capacitary measures and their use in optimization problems we refer to the book \cite{bubu05}.

Here we notice that the admissible capacitary measures obtained as limits of sequences $(V_n)$ of potentials in $\V$ are the measures $\mu$ such that their absolutely continuous part $\mu^a$ with respect to the Lebesgue measure belong to $\V$. Let us denote by $\overline\V$ this relaxed class of measures. The relaxed problem associated to \eqref{minpb} is then
\be\label{relpb}
\min\bigg\{\int_D g(x)u(x)\,dx\ :\ -\Delta u+\mu u=f,\ u\in H^1_0(D)\cap L^2(\mu),\ \mu\in\overline\V\bigg\},
\ee
where the precise meaning of the state equation has to be intended in the weak form
$$\int_D\nabla u\nabla\phi\,dx+\int_D u\phi\,d\mu=\int_D f\phi\,dx\qquad\forall\phi\in H^1_0(D)\cap L^2(\mu).$$
It is convenient to introduce the resolvent operator $R_\mu$ associated to the operator $-\Delta+\mu$; it is well known that $R_\mu$ is self-adjoint on $L^2(D)$.

Since the class of capacitary measures is known to be compact with respect to the $\gamma$ convergence, the relaxed problem \eqref{relpb} admits a solution $\mu\in\overline\V$. We aim to show that we can actually find a solution in the original admissible class $\V$.

\begin{lemm}\label{l1}
Let $\mu\in\overline\V$ be a solution of the relaxed optimization problem \eqref{relpb}. Then
\be\label{rmu}
R_\mu(g)R_\mu(f)\le0\qquad\hbox{a.e. on }D.
\ee
\end{lemm}

\begin{proof}
For every $\eps>0$ let $\mu_\eps=\mu+\eps\phi$, where $\phi$ is a generic continuous nonnegative function. Since $\Psi$ is decreasing, the capacitary measure $\mu_\eps$ still belongs to the relaxed admissible class $\overline\V$, and so
\be\label{eq1}
\int_D g(x)u(x)\,dx\le\int_D g(x)u_\eps(x)\,dx
\ee
where $u$ and $u_\eps$ respectively denote the solutions of
\[\begin{split}
&-\Delta u+\mu u=f\hbox{ in }D,\qquad u\in H^1_0(D)\cap L^2(\mu),\\
&-\Delta u_\eps+\mu_\eps u_\eps=f\hbox{ in }D,\qquad u_\eps\in H^1_0(D)\cap L^2(\mu_\eps).
\end{split}\]
Then, setting $w_\eps=(u_\eps-u)/\eps$, we have
\be\label{eq2}
-\Delta w_\eps+\mu w_\eps=-\phi u_\eps
\ee
and by \eqref{eq1}
$$\int_D g(x)w_\eps(x)\,dx\ge0\;.$$
Since $\mu_\eps$ is $\gamma$-converging to $\mu$ we have that $u_\eps$ tends to $u$ in $L^2(D)$ and, by \eqref{eq2} we obtain that $w_\eps$ tends to $w$ in $L^2(D)$, where $w$ solves
$$-\Delta w+\mu w=-\phi u\hbox{ in }D,\qquad u\in H^1_0(D)\cap L^2(\mu)\;.$$
Since the resolvent operator $R_\mu$ of $-\Delta+\mu$ is self-adjoint, we have
$$0\le\int_D g(x)w(x)\,dx=-\int_D R_\mu(g)\phi u\,dx$$
which gives, since $\phi$ is arbitrary,
$$R_\mu(g) u\le0\qquad\hbox{a.e. on }D,$$
which is the conclusion \eqref{rmu}.
\end{proof}

\begin{lemm}\label{l2}
Assume that $g\ge0$ and let $\mu$ be a solution of the relaxed optimization problem \eqref{relpb}. Let $\nu\in\overline\V$ be another capacitary measure, with $\nu\le\mu$. Then
$$\int_D gR_\nu(f)\,dx\le\int_D gR_\mu(f)\,dx\;.$$
\end{lemm}

\begin{proof}
The functions $u=R_\mu(f)$ and $v=R_\nu(f)$ respectively solve the PDEs
\[\begin{split}
&-\Delta u+\mu u=f\;,\\
&-\Delta v+\nu v=f\;.
\end{split}\]
Then we have
$$-\Delta(u-v)+\nu(u-v)=-u(\mu-\nu)$$
so that $u-v=R_\nu\big(-u(\mu-\nu)\big)$. Hence,
$$\int_D g(u-v)\,dx=\int_D gR_\nu\big(-u(\mu-\nu)\big)\,dx=-\int_D R_\nu(g)u(d\mu-d\nu)$$
where in the last equality we used the fact that $R_\nu$ is self-adjoint. Now, since $g\ge0$, by the maximum principle we have $R_\nu(g)\ge0$ and $R_\mu(g)\ge0$; then by Lemma \ref{l1} we have $u=R_\mu(f)\le0$ and so
$$\int_D g(u-v)\,dx\ge0,$$
as required.
\end{proof}

\begin{rema}\label{rema1}
It is easy to check that the same proof also works if we assume $g\le0$. Moreover, using the fact that the resolvent operators are self-adjoint, the relaxed optimization problem \eqref{relpb} can be written also in the form
$$\min\bigg\{\int_D fR_\mu(g)\,dx\ :\ \mu\in\overline\V\bigg\}$$
and so we can also assume $f\ge0$ (or $f\le0$) with no sign assumtion on $g$, and obtain for $\nu\le\mu$
$$\int_D gR_\nu(f)\,dx\le\int_D gR_\mu(f)\,dx\;.$$
\end{rema}

We are now in a position to prove the existence of an optimal potential in the original class $\V$.

\begin{theo}\label{mainPhi}
Let $D\subset\R^d$ be a bounded open set and let $\Psi$ satisfy the assumptions i) and ii) above. Then, for every $f,g\in L^2(D)$ with $g\ge0$, the original optimization problem \eqref{minpb} has a solution.
\end{theo}

\begin{proof}
Let $V_n\in\V$ be a minimizing sequence for the optimization problem \eqref{minpb}. Then, $v_n=\big(\Psi(V_n)\big)^{1/p}$ is a bounded sequence in $L^p(D)$ and so, up to a subsequence, $v_n$ converges weakly in $L^p(D)$ to some function $v$. We prove that the potential $V=\Psi^{-1}(v^p)$ is a solution to \eqref{minpb}. By its definition we have $V\in\V$ and so it remains to prove that
$$\int_D gR_V(f)\,dx\le\liminf_n\int_D gR_{V_n}(f)\,dx\;.$$
Since the $\gamma$-convergence is compact, we may assume that, up to a subsequence, $V_n$ $\gamma$-converges to a capacitary measure $\mu\in\overline\V$, which implies
$$\int_D gR_\mu(f)\,dx=\lim_n\int_D gR_{V_n}(f)\,dx\;.$$
Therefore, it remains only to prove the inequality
\be\label{ineqV}
\int_D gR_V(f)\,dx\le\int_D gR_\mu(f)\,dx\;.
\ee
By the definition of $\gamma$-convergence, we have that for any $u\in H^1_0(D)$, there is a sequence $u_n\in H^1_0(D)$ which converges to $u$ in $L^2(D)$ and is such that
\[\begin{split}
\int_D|\nabla u|^2\,dx+\int_D u^2\,d\mu
&=\lim_{n\to\infty}\int_D|\nabla u_n|^2\,dx+\int_D u_n^2 V_n\,dx\\
&=\lim_{n\to\infty}\int_D|\nabla u_n|^2\,dx+\int_D u_n^2\Psi^{-1}(v_n^p)\,dx\\
&\ge\int_D|\nabla u|^2\,dx+\int_D u^2\Psi^{-1}(v^p)\,dx\\
&=\int_D|\nabla u|^2\,dx+\int_D u^2V\,dx,
\end{split}\]
where the inequality above is due to the strong-weak lower semicontinuity of integral functionals (see for instance \cite{busc}), which follows by the assumptions made on the function $\Psi$. Thus, for any $u\in H^1_0(D)$, we have 
$$\int_D u^2\,d\mu\ge\int_D u^2V\,dx,$$
which gives $V\le\mu$. The inequality \eqref{ineqV} now follows by Lemma \ref{l2}.
\end{proof}

\begin{rema}\label{rema2}
By Remark \ref{rema1} the same conclusion holds if $g\le0$, and also if $f\ge0$ (or $f\le0$) and no sign assumption on $g$.
\end{rema}

\section{Necessary conditions of optimality}\label{snec}

We assume in this section that the conditions above are satisfied, so that an optimal potential $V$ exists. In general, the constraint $\int_D\Psi(V)\,dx\le1$ is not always saturated, since the data $f$ and $g$ may change sign. Therefore passing to the problem with a Lagrange multiplier
$$\min\bigg\{\int_D gR_V(f)\,dx+\lambda\int_D\Psi(V)\,dx\bigg\}$$
we intend that $\lambda=0$ when $\int_D\Psi(V)\,dx<1$. We assume that the function $\Psi$ is differentiable, and we write the variations on $V$ and on $u$ as
$$V+\eps V',\qquad u+\eps u'.$$
We then obtain
$$\int_D gu\,dx+\lambda\int_D\Psi(V)\,dx\le\int_D g(u+\eps u')\,dx+\lambda\int_D\Psi(V+\eps V')\,dx.$$
An easy computation gives that for $\eps$ small we have
$$-\Delta u'+Vu'=-V'u$$
so that $u'=R_V(-V'u)$. Using the fact that the resolvent operator $R_V$ is self-adjoint, we deduce that for every $V'$
$$0\le\int_D gu'\,dx+\lambda\int_D\Psi'(V)V'\,dx=\int_D\big(-R_V(g)u+\lambda\Psi'(V)\big)V'\,dx.$$
Since $V'$ is arbitrary, we obtain
$$\begin{cases}
uR_V(g)=\lambda\Psi'(V)&\hbox{on the set }\{V>0\}\\
uR_V(g)\le\lambda\Psi'(V)&\hbox{on the set }\{V=0\}.
\end{cases}$$

Note that in the case of not saturated constraint we have $\lambda=0$, so that the necessary conditions above give
$$\begin{cases}
uR_V(g)=0&\hbox{on the set }\{V>0\}\\
uR_V(g)\le0&\hbox{on the set }\{V=0\}.
\end{cases}$$
In particular, when $g\ge0$ (not identically zero) we have $R_V(g)>0$ where $V$ is finite, so that the conditions above simply give
$$\begin{cases}
u=0&\hbox{on the set }\{V>0\}\\
u\le0&\hbox{on the set }\{V=0\}.
\end{cases}$$

\section{Some numerical simulations}\label{snumer}

In this section we present and show a numerical method in order to solve a problem of the kind of \eqref{minpb}.

We start showing as to get a gradient descent direction. Later we describe an algorithm for the optimization problem and finally we show some numerical experiments for some functions $g$ and different choices of the function $f$ which have non-constant sign, and diverse functions $\Psi(V)=\exp(-\alpha V)/m$ for different values of $\alpha>0$ and $m\in(0,1)$ in order to impose different volume constraints. 

\subsection{The descent direction}

Our goal is to solve numerically minimization problems of the form \eqref{minpb}:

\be\label{eq:num_cost}
\min\int_D g(x)u(x)\,dx
\ee
subject to
\be\label{eq:state}
\begin{cases}
-\Delta u+V u=f&\hbox{ in }D,\\
u=0&\hbox{ on }\partial D
\end{cases}\ee
\be\label{vol_con}
\int_D e^{-\alpha V(x)}\,dx\le m,
\ee
and where the optimal potential $V:D\to[0,+\infty]$ is a Lebesgue measurable function. In the case of shape optimization problems a domain $\O\subset D$ is associated to the potential
$$V(x)=\begin{cases}
0&\hbox{ if }x\in\O,\\
+\infty&\hbox{ if }x\in D\setminus\O,
\end{cases}$$
so that
$$|\O|=\int_D e^{-\alpha V(x)}\,dx.$$

Let us assume $V$ and $V'$ two admissible potentials and let us compute \emph{formally} the derivative of cost function
$$I(V)= \int_D g(x)u(x)\,dx$$
at the position $V$ in the direction $V'$. Under appropriate regularity hypotheses on $V$ and the associated state $u$, the first derivative of the cost functional \eqref{eq:num_cost} with respect to $V$ in any direction $V'$ exists and takes the form:
\be
\frac{d I(V)}{d V}\cdot V' =
\int_{D} V'(x)u(x)p(x)\,dx ,
\label{deriveE}
\ee
where $p$ is the unique solution of the adjoint equation
\be\label{adjoint_eq}
\begin{cases}
-\Delta p+V p=-g&\hbox{ in }D,\\
p=0&\hbox{ on }\partial D.
\end{cases}
\ee

For any $\eta\in\R^+$, $\eta\ll1$ we denote by $u^\eta$ the solution of \eqref{eq:state} for $V^\eta=V+\eta V'$, we would like to compute
\be\label{derive}
\frac{dI(V)}{dV}\cdot V'=
\lim_{\eta\to0}\frac{I(V+\eta V')-I(V)}{\eta}.
\ee
In this way, we put $u^\eta=u+\eta y^\eta$, where $y^\eta$ is the solution of
\be\label{eq:u_nu}
\begin{cases}
-\Delta y^\eta+V^\eta y^\eta=-V'u&\hbox{ in } D,\\
y^\eta=0&\hbox{ on }\partial D.
\end{cases}\ee
From \eqref{derive} and \eqref{adjoint_eq} we have
\[\begin{split}
\frac{I(V+\eta V')-I(V)}{\eta}&=\frac{1}{\eta}\int_D g(x) u^\eta(x)-g(x)u(x)\,dx\\
&=\int_Dg(x)y^\eta(x)\,dx\\
&=-\int_D\nabla p(x)\nabla y^\eta(x)\,dx-\int_DV(x)p(x)y^\eta(x)\,dx.
\end{split}\]
On the other hand, from the above formula and having in mind \eqref{eq:u_nu} and the  we arrive to
$$\ba{l}
\dis\frac{d I(V)}{d V}\cdot V' =
\lim_{\eta\to0}\left( \int_D V'(x)u(x)p(x)\,dx-\eta\int_DV'(x)y^\eta(x) p(x)\,dx\right)\\\dis= \int_D V'(x)u(x)p(x)\,dx.
\ea$$
Then, taking into account formula \eqref{deriveE}, in order to apply a gradient descent method it is enough to take the direction
$$V'(x)=-u(x)p(x).$$

In order to take into account the volume constraint \eqref{vol_con} on $V$, we introduce the Lagrange multiplier $\lambda\in\R$ and the functional
$$I_\lambda(V)=I(V)+\lambda\int_D e^{-\alpha V(x)}\,dx$$
and therefore,
\be\label{d_d2}
\frac{dI_\lambda(V)}{dV}\cdot V'=
\int_D V'(x)(u(x)p(x)-\lambda\alpha e^{-\alpha V(x)})\,dx.
\ee
where the multiplier $\lambda$ is determined in order to assure \eqref{vol_con}.

Thus, a general gradient algorithm to solve numerically the extremal problem \eqref{eq:num_cost} - \eqref{eq:state} - \eqref{vol_con} is the following.
\begin{itemize}
\item Initialization: choose an admissible $V_0$;
\item for $k\ge0$, iterate until convergence as follows:
\begin{itemize}
\item compute $u_k$ solution of \eqref{eq:state} and $p_k$ solution of \eqref{adjoint_eq}, both corresponding to $V=V_k$;
\item compute the associated descent direction $V'_k$ given by \eqref{d_d2} associated to $u_k$ and $p_k$;
\item update the potential $V_k$:
$$V_{k+1}=V_k+\eta_kV'_k,$$
with $\eta_k$ small enough to ensure the decrease of the cost function.
\end{itemize}
\end{itemize}

\subsection{Numerical Simulations}

For our numerical experiments we decided to use the free software FreeFEM++ v 3.50 (see {\tt http://www.freefem.org/}, see \cite{FF}), complemented with the library NLopt (see {\tt http://ab-initio.mit.edu/wiki/index.php/NLopt}) using the Method of Moving Asymptotes as the optimizing routing (see \cite{svanberg87}). This technique is a gradient method based on a spatial type of convex approximation where in each iteration a strictly convex approximation subproblem is generated and solved. For the implementation of this algorithm the main required data are the initialization $V_0$, the associated routines to the cost and volume function and the associated routines to the gradient of the cost and volume function using the adjoint state. The admissible potentials $V$ take values in $[0,+\infty]$ but from the numerical point of view it is advisable to constrain $V$ to take values on a bounded interval $[0,V_{max} ]$, with $V_{max}$ large enough. These data are required for the algorithm too. We observe that, when $V$ takes its maximal value $V_{max}$, the state $u$ is very small and practically vanishes, according to the well-posed character of the extremal problem and the state equation. This is consistent with the necessary conditions of optimality obtained in Section \ref{snec}.

We show the numerical result for some experiments. We have made the simulation in the two dimensional case and we have chosen $D=(0,1)\times(0,1)$. The optimization criterion we consider is the minimization of the average solution $u=R_V(f)$ on $D$ for a given right-hand side $f$, where the potential $V$ varies in the admissible class
$$\V=\left\{V\ge0,\ \int_D e^{-\alpha V(x)}\,dx\le m\right\}.$$
Therefore, in the following we take $g=1$ and we consider various choices for $f$ and for the parameters $\alpha$ and $m$. It has to be noticed that, if $f\ge0$, by the maximum principle all the solutions $u$ are nonnegative, so that the optimization problem has the trivial solution $V=+\infty$ for which the corresponding state is $u=0$.

We use a $P_2$-Lagrange finite element approximations for $u$ and $p$ solutions of the state and costate equations (\ref{eq:state}) and (\ref{adjoint_eq}) respectively, and $P_0$-Lagrange finite element approximations for the potential $V$. In our simulations we have considered $V_{max}=10^4$ and a regular mesh of $200\times200$ elements, see Figure \ref{mesh}. We analyze different cases. For the optimal potential representation we use a grey scale, where black corresponds to 0 value and white to $V_{max}$.

\begin{figure}[!t]
\begin{center}
\includegraphics[width=15cm]{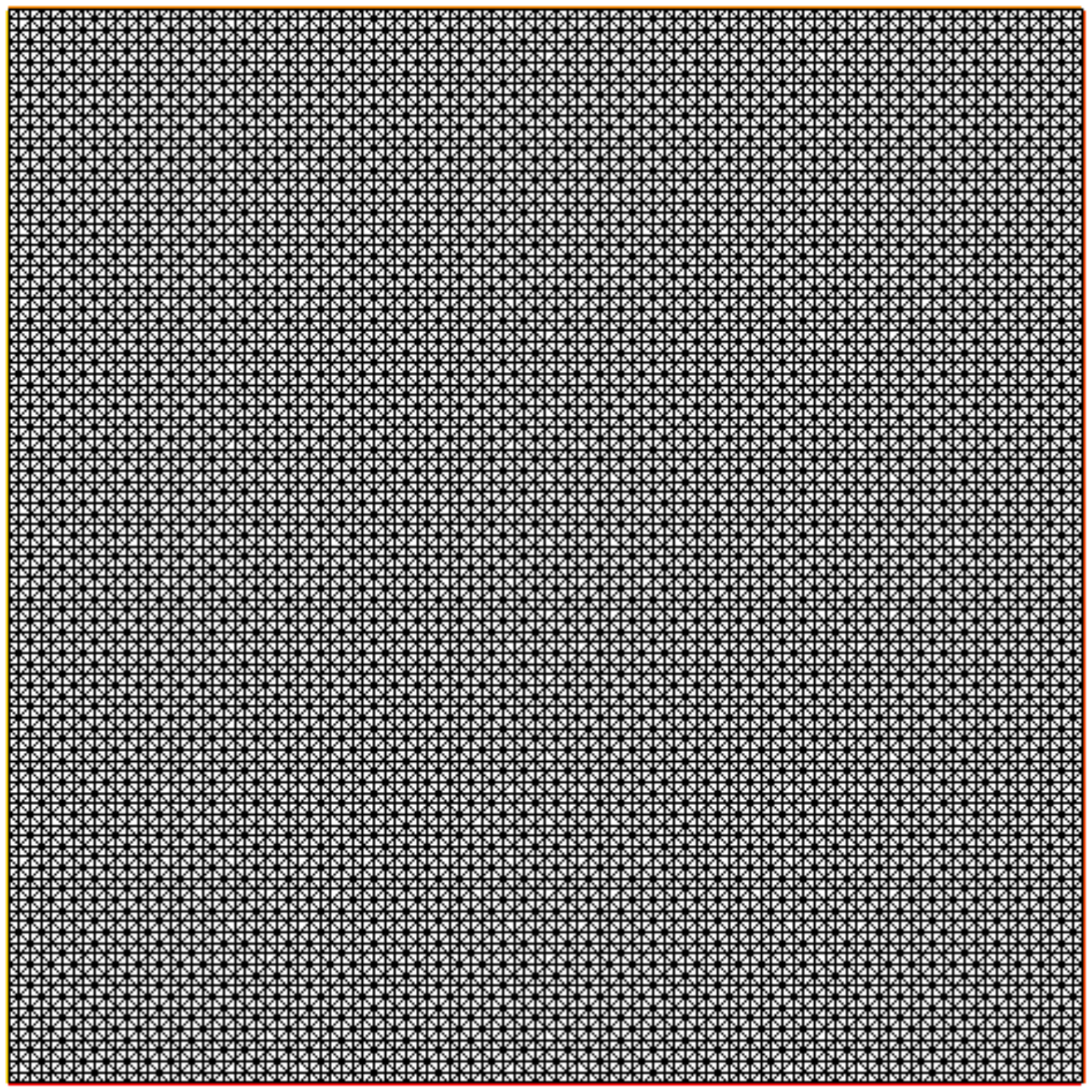}
\end{center}
\vskip-1.00cm
\caption{The domain $D$ and its triangulation. Number of nodes: 40401. Number of triangles: 80000.}\label{mesh}
\end{figure}

The first case we consider is when $f(x,y)=-(1+10x)$ (see Figure \ref{F6}) and $m=0.2$. We expect that the optimal potential consists of a quasi-ellipsoid-shape placed on the region where the values of the function $f$ are smaller. For this case we make two different experiments for various values of the parameter $\alpha$ related to the volume constraint. In Figure \ref{F7} left, we have used $\alpha=0.09$ while in Figure \ref{F7} right we have used $\alpha=3.10^{-4}$. We can observe that in the first case the optimal potential $V_{opt}$ is distributed on all the domain $D$, while in the second case (when $\alpha$ is small enough) the optimal potential is very close to an optimal shape.

\begin{figure}[!t]
\begin{center}
\includegraphics[width=12cm]{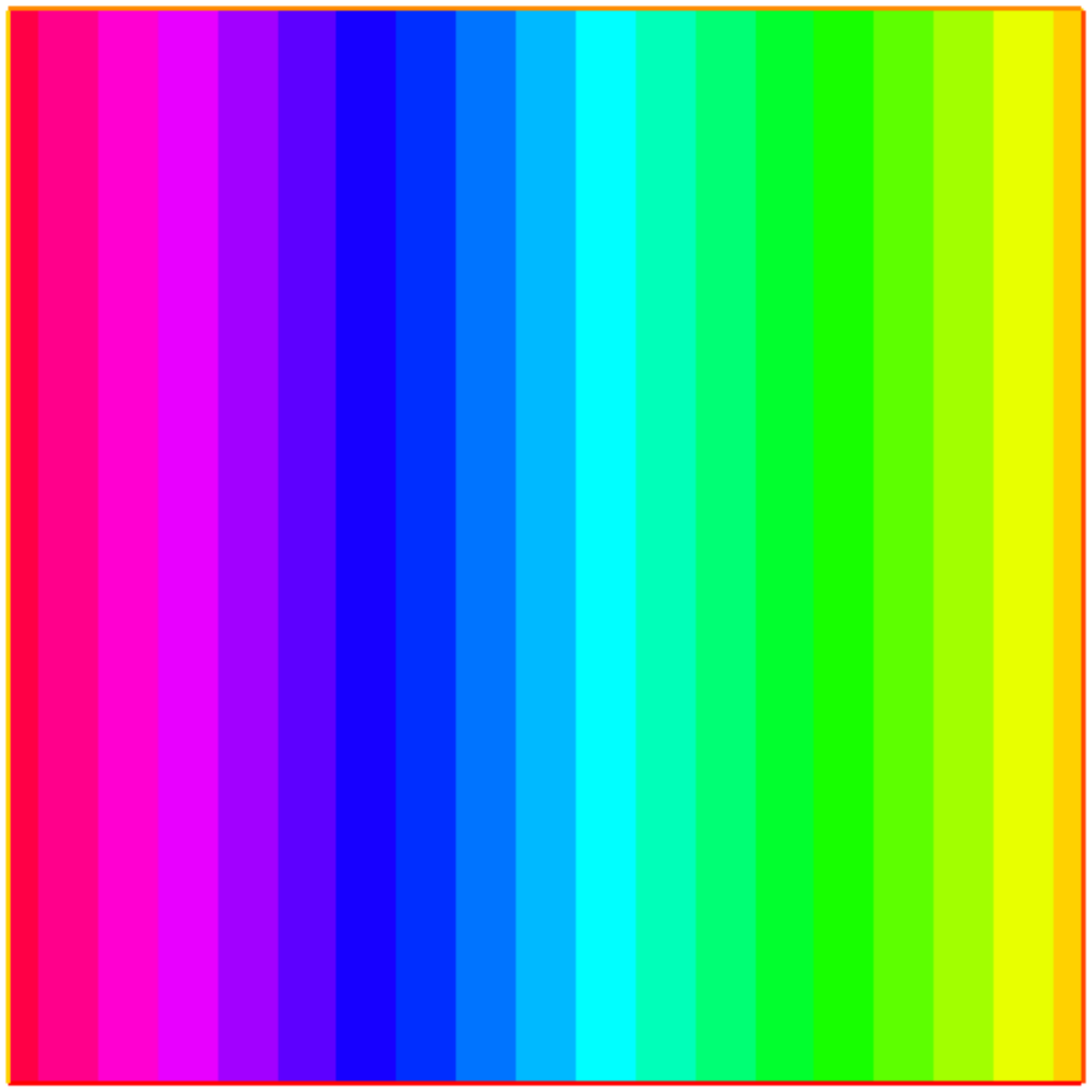}
\end{center}
\caption{The right hand side function $f(x,y)=-(1+10x)$ }\label{F6}
\end{figure}

In the subsequent numerical experiments we fix $\alpha=3.10^{-4}$ in order to recover optimal shapes, and we consider various functions $f$ for the right hand-side of the state equation, where $f$ changes its sign.

\begin{figure}[!t]
\begin{minipage}[!t]{18cm}
\centering
\hspace{-3cm}
\includegraphics[width=11.5cm]{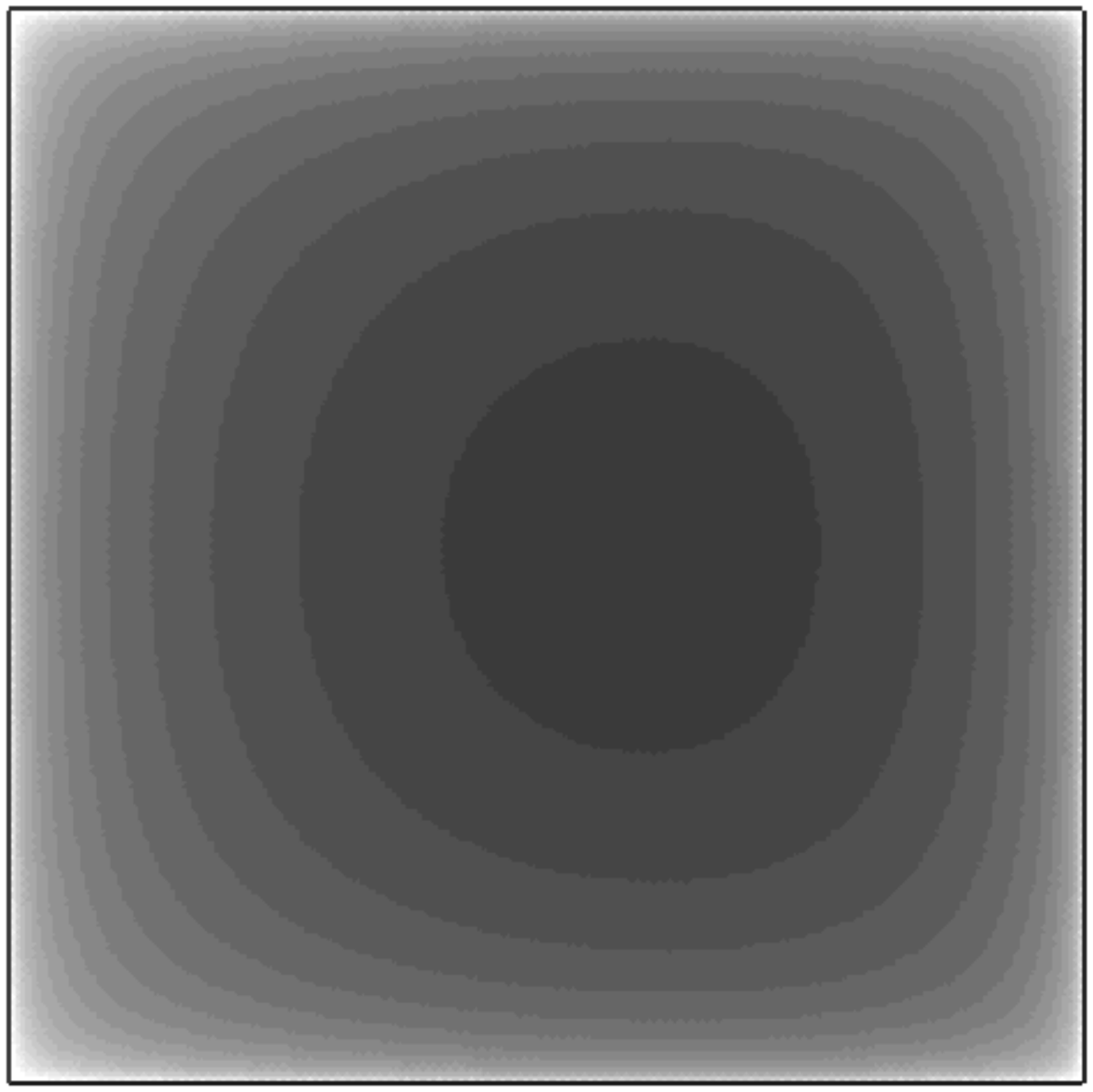}
\hspace{-3.2cm}
\includegraphics[width=11.5cm]{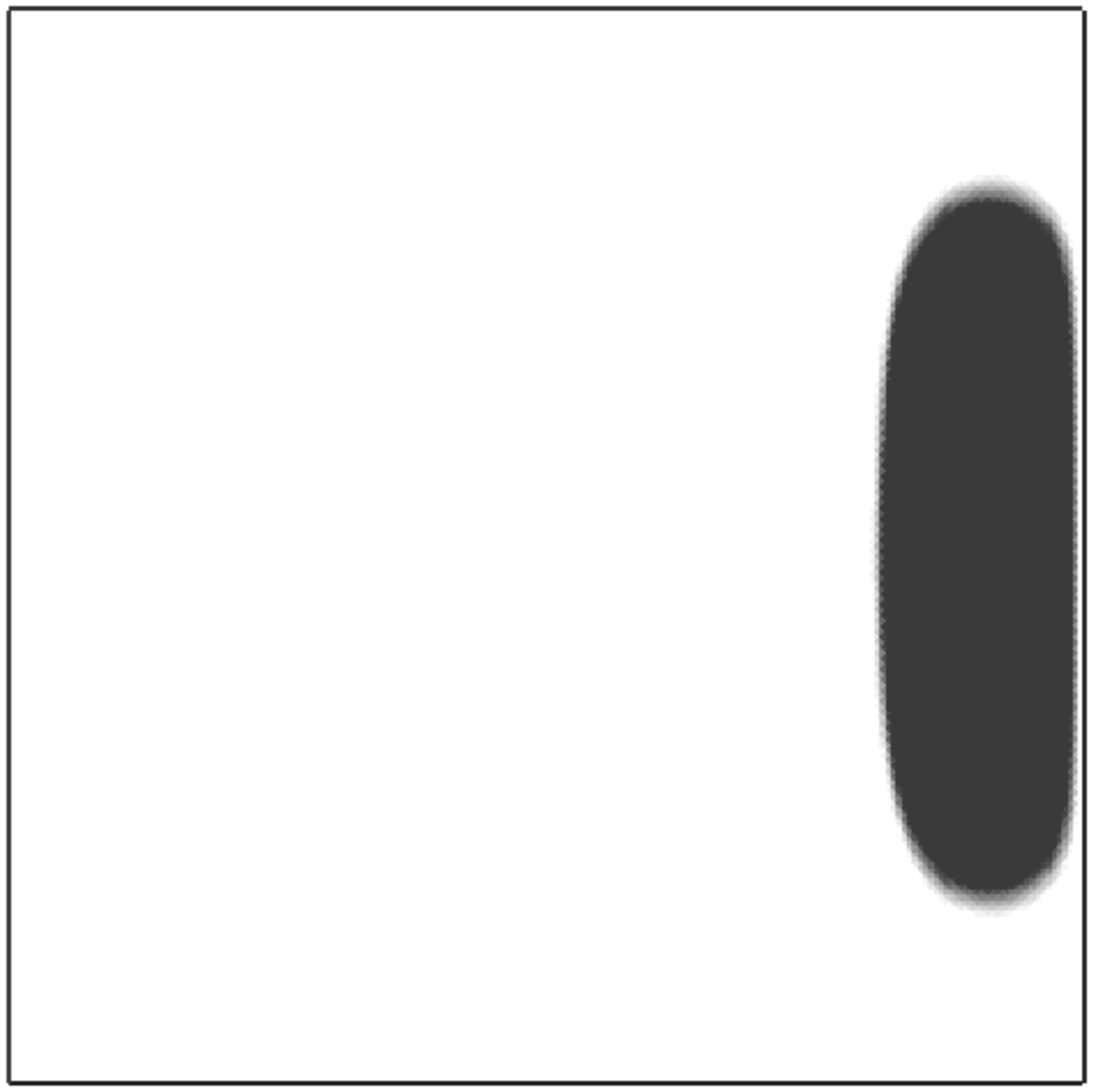}
\end{minipage}\vskip-0.60cm
\caption{Example 1 -- The optimal potential $V_{opt}$ for volume contraint $m=0.2=m_{opt}$. Case $\alpha=0.09$ (left) and $\alpha=3.10^{-4}$ (right).} \label{F7}
\end{figure}

For the Example 2 we consider the right hand-side function:
$$f(x,y)=\begin{cases}
-1&\hbox{if }y-1.4x\ge0.3\\
1&\hbox{if }y-1.4x<0.3
\end{cases}$$
negative on a corner of the domain D, and positive on the rest (see Figure \ref{F4}). In this case, we make two simulation with volume constraints $m=0.2$ (small volume) and $m=0.45$ (larger volume). In both cases we observe that the optimal shapes are placed near the corner where the function $f$ is negative (see Figure \ref{F3}). However, in the case of small volume constraint the optimal domain $\O_{opt}$ has volume equal to $m$ (saturation of the constraint, see Figure \ref{F3} left), while in the case of larger $m$ the optimal domain satisfies $|\O_{opt}|<m$ (see Figure \ref{F3} right). For instance, in the case under consideration, the optimal domain uses only $0.33276$ of the volume, of the $0.45$ available.

\begin{figure}[!t]
\begin{center}
\includegraphics[width=12cm]{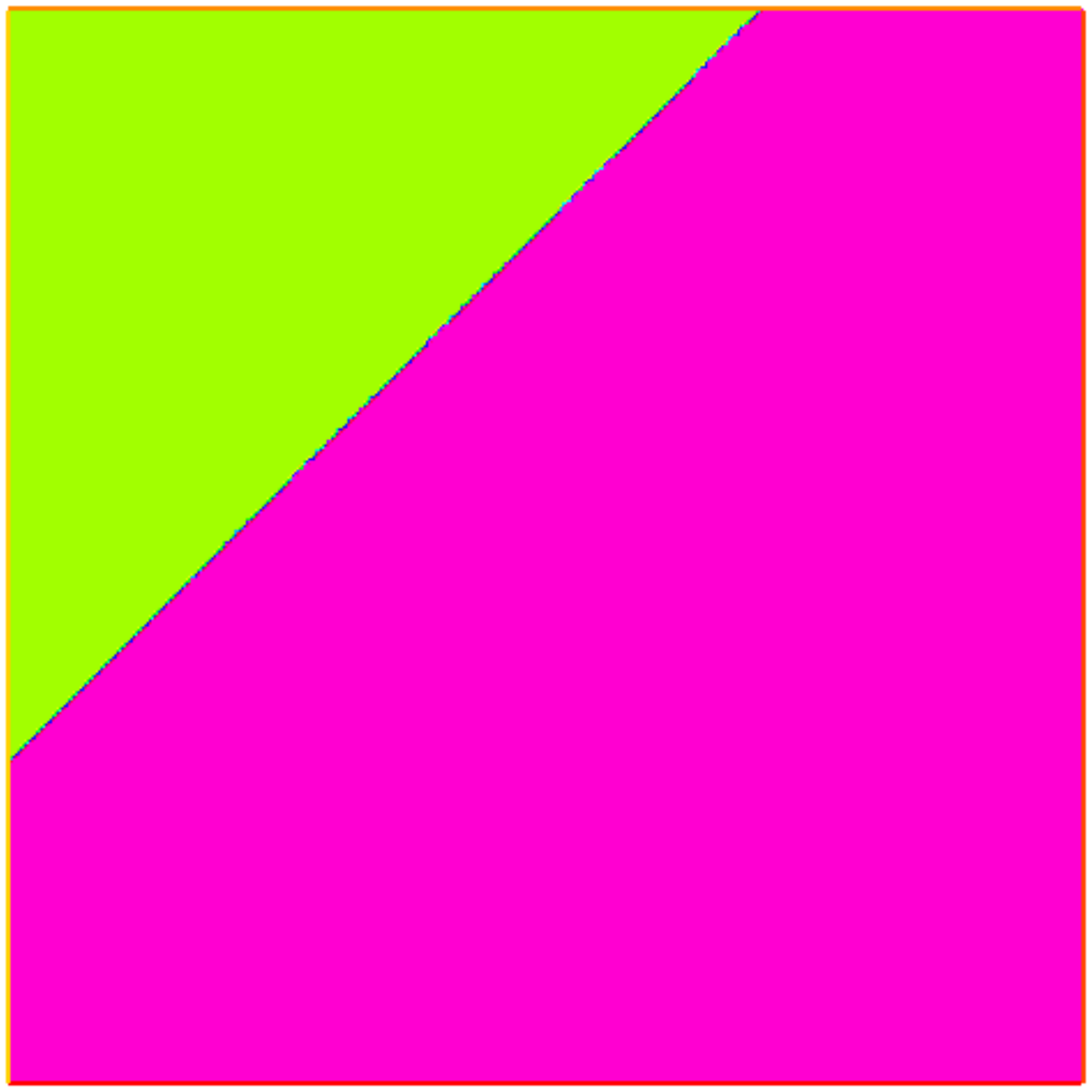}
\end{center}\vskip-0.7cm
\caption{The right hand side function $f(x,y)=-1$ if $y-1.4x\ge 0.3$, and $f(x,y)=1$ if $y-1.4x< 0.3$}\label{F4}
\end{figure}

\begin{figure}[!t]
\begin{minipage}[!t]{18cm}
\centering
\hspace{-3cm}
\includegraphics[width=11.8cm]{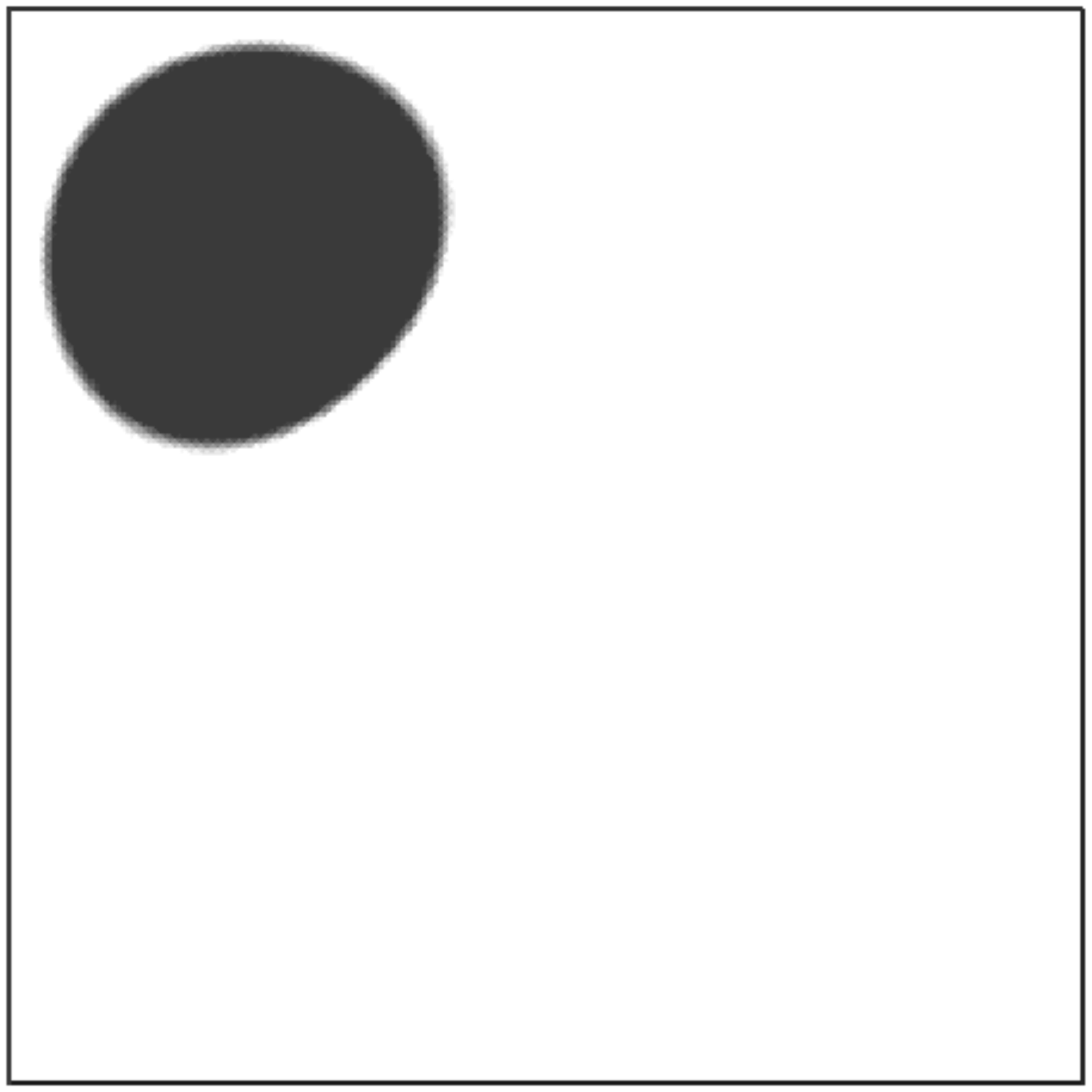}
\hspace{-3.2cm}
\includegraphics[width=11.5cm]{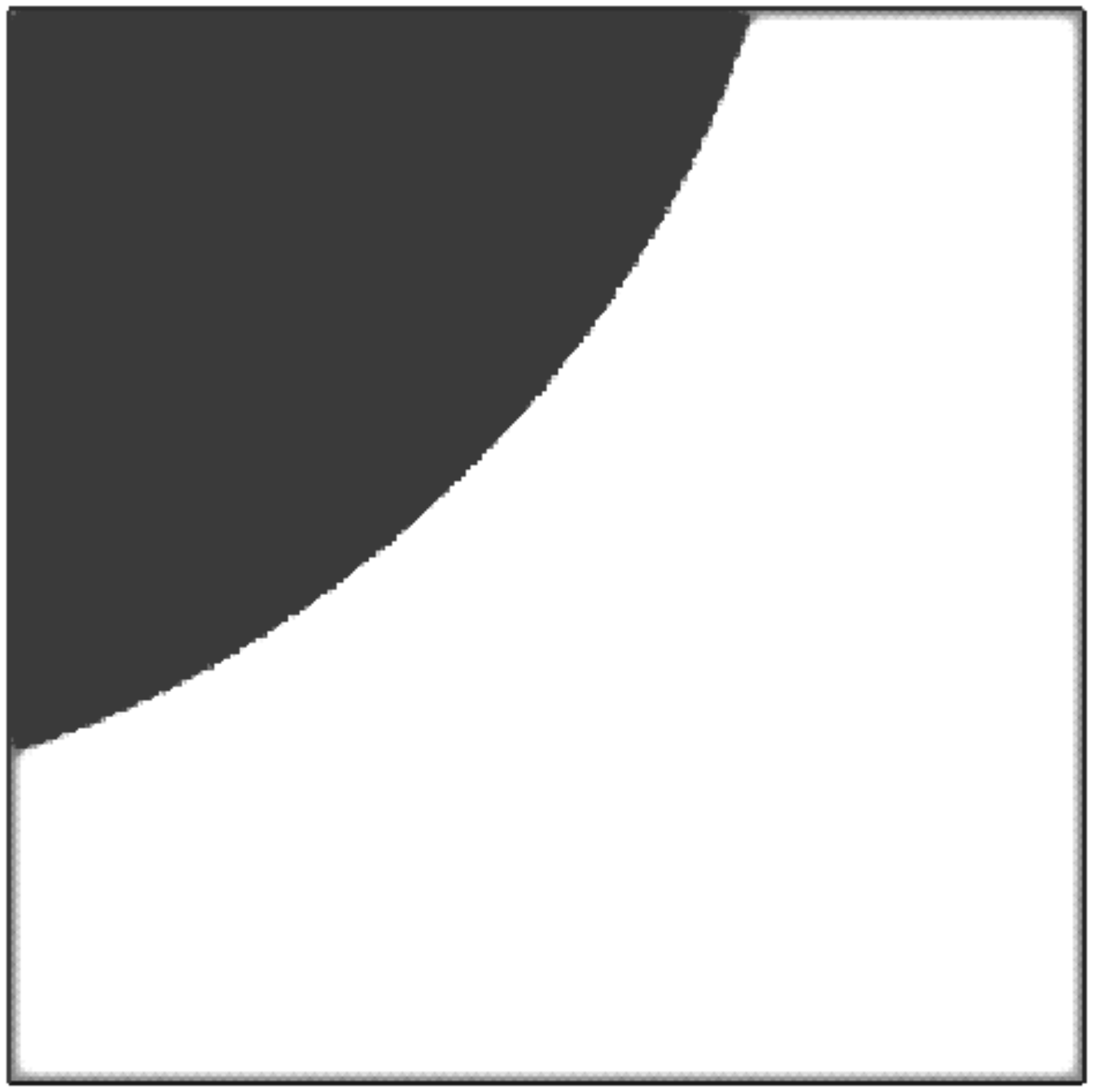}
\end{minipage}\vskip-0.7cm
\caption{Example 2 -- Optimal potential $V_{opt}$. Case: $m=0.2$ (left), $m=0.45$ occupied volume $0.33276$ (right).} \label{F3}
\end{figure}

For the Example 3 the right hand-side function which we consider is a characteristic function which takes the values $1$ on a centered non-symmetric cross and $-1$ on the rest of the domain $D$ (see Figure \ref{F1} left). In this case we have imposed a volume constraint $m=0.45$ and we observe (see Figure \ref{F1} right) that the optimal shape is made of four small balls of different sizes at the corners of the square domain outside of the cross and the volume constraint is saturated.

Finally, in the Example 4 we consider for the right hand-side $f$ the reverse case of the Example 3. We consider a characteristic function where on a centered non-symmetric cross takes the value $-1$ and $1$ on the rest of the domain (see Figure \ref{F2} left). For this simulation the results give an optimal shape that is placed around the cross, including regions where $f$ is negative but also small areas around the cross where $f$ is positive. The volume constraint in this case is not saturated using $0.378404$ of the $m=0.5$ available.

\begin{figure}[!t]
\begin{minipage}[!t]{18cm}
\centering
\hspace{-3cm}
\includegraphics[width=11.5cm]{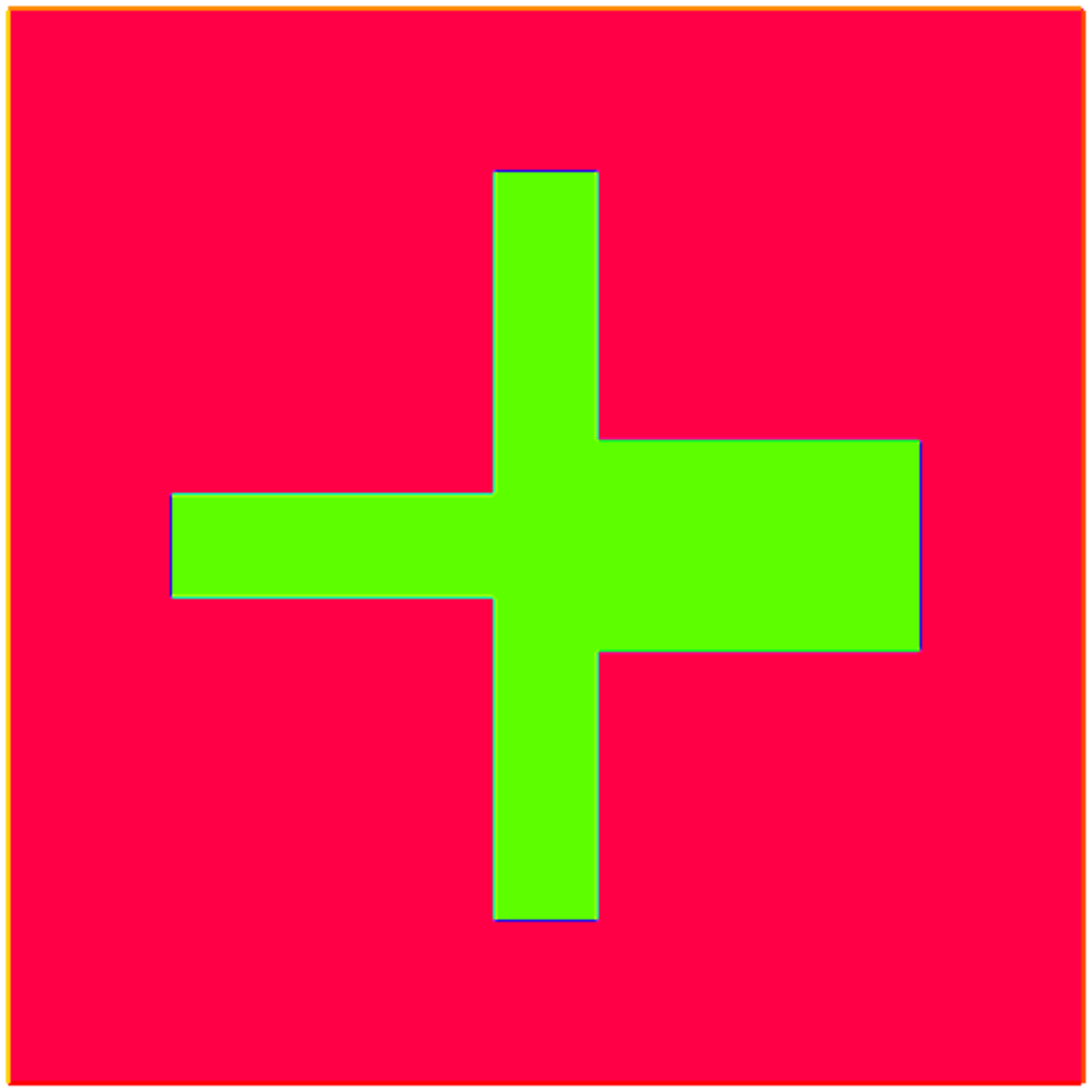}
\hspace{-3.2cm}
\includegraphics[width=11.5cm]{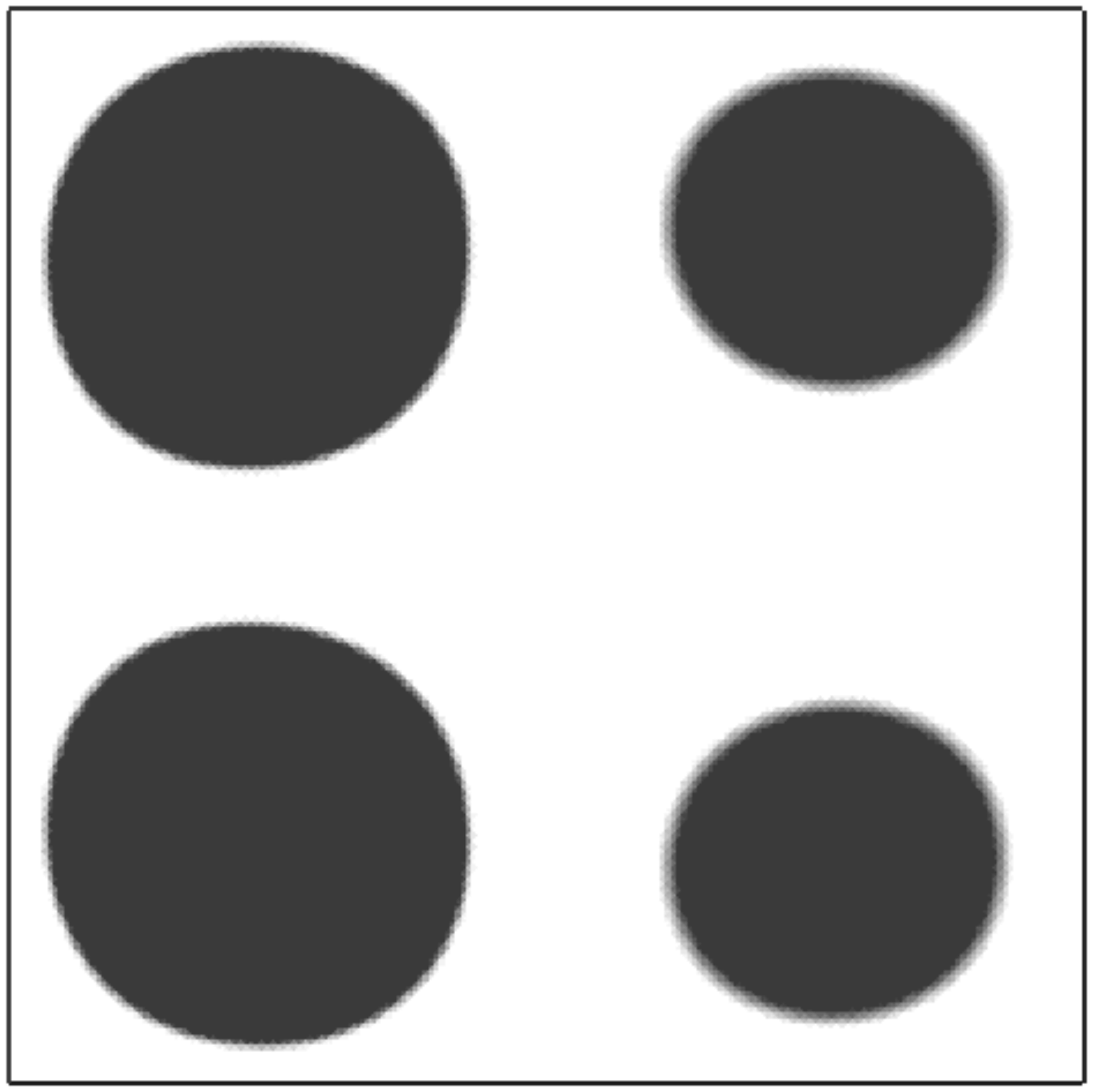}
\end{minipage}\vskip-0.7cm
\caption{Example 3 -- The right hand side function $f$ (left) and the optimal potential $V_{opt}$ (right). The volume $m=0.45$ is all occupied.} \label{F1}
\end{figure}

\begin{figure}[!t]
\begin{minipage}[!t]{18cm}
\centering
\hspace{-3cm}
\includegraphics[width=11.5cm]{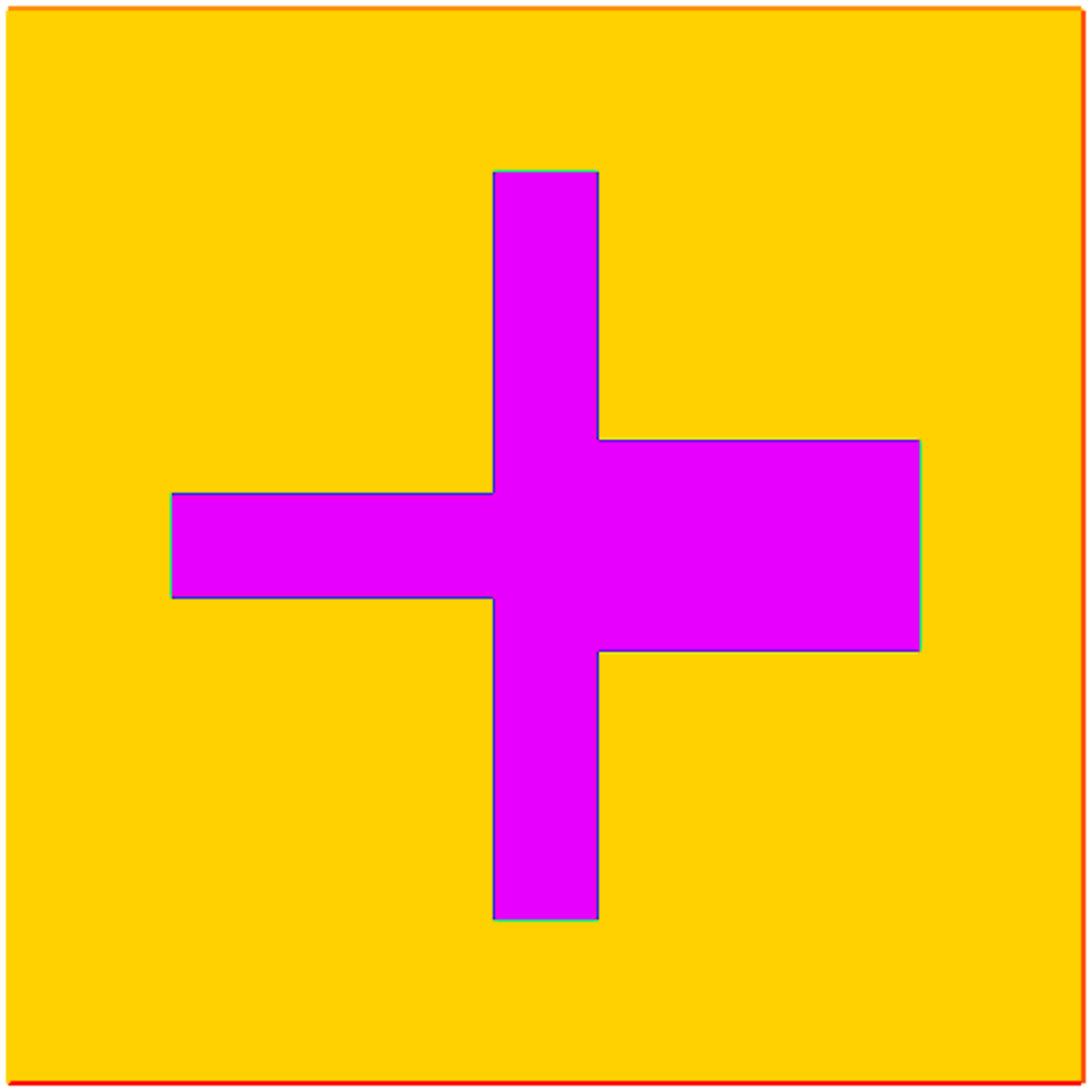}
\hspace{-3.2cm}
\includegraphics[width=11.5cm]{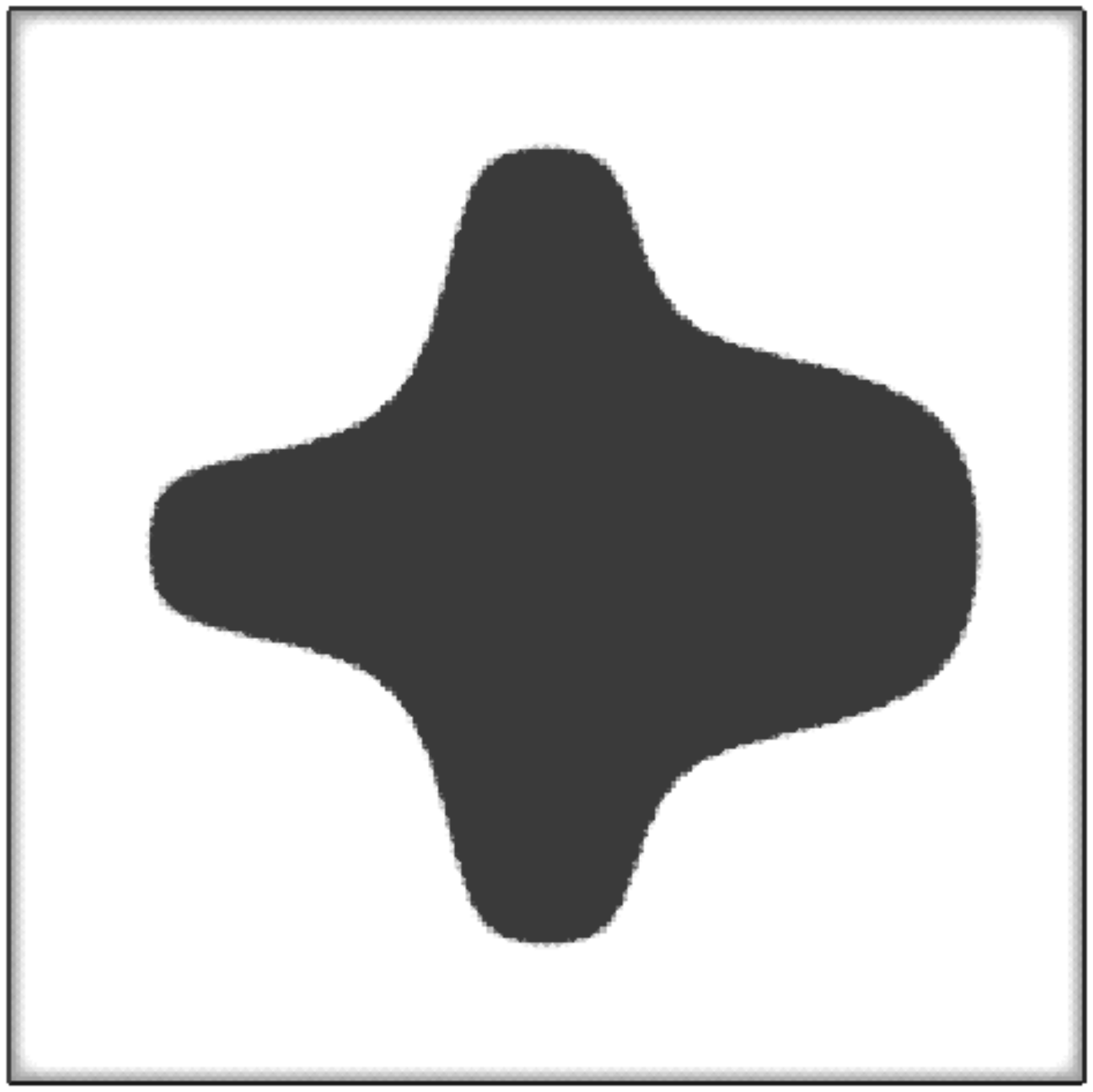}
\end{minipage}\vskip-0.7cm
\caption{Example 4 -- The right hand side function $f$ (left) and the optimal potential $V_{opt}$ (right). The occupied volume is $0.378404$ of the $m=0.5$ available.}\label{F2}
\end{figure}

In conclusion, according to the previous results we have shown the numerical evidence that the optimization problems in the form of (\ref{minpb}) admit optimal solutions when the data $f$ and $g$ are allowed to change sign. We can observe that in order to approximate the shape optimization problem with Dirichlet condition on the free boundary, taking the function $\Psi(s)=e^{-\alpha s}$, with $\alpha$ small enough, is a good choice in order to achive optimal shapes. Moreover, we can observe that the optimal shapes are located mostly in areas where the sign of $f$ is negative but they may in some cases occupy also small regions where $f$ is positive. Finally, the optimal domains may not always saturate the volume constraint.

\bigskip
\ack This work started during a visit of the second author at the Department of Mathematics of University of Pisa and continued during a stay of the authors at the Centro de Ciencias de Benasque ``Pedro Pascual''. The authors gratefully acknowledge both Institutions for the excellent working atmosphere provided. The work of the first author is part of the project 2015PA5MP7 {\it``Calcolo delle Variazioni''} funded by the Italian Ministry of Research and University. The first author is member of the Gruppo Nazionale per l'Analisi Matematica, la Probabilit\`a e le loro Applicazioni (GNAMPA) of the Istituto Nazionale di Alta Matematica (INdAM). The second author has been partially supported by FEDER and the Spanish Ministerio de Econom\'ia y Competitividad project MTM2014-53309-P.

\bigskip
{\small\noindent
Giuseppe Buttazzo:\\
Dipartimento di Matematica,
Universit\`a di Pisa\\
Largo B. Pontecorvo 5,
56127 Pisa - ITALY\\
{\tt buttazzo@dm.unipi.it}\\
{\tt http://www.dm.unipi.it/pages/buttazzo/}

\bigskip\noindent
{Faustino Maestre Caballero:\\
Dpto. Ecuaciones Diferenciales y An\'alisis Num\'erico, 
Universidad de Sevilla\\
C/ Tarfia s/n. Aptdo 1160, 
41080 Sevilla - SPAIN}\\
{\tt fmaestre@us.es}\\
{\tt http://personal.us.es/fmaestre}

\bigskip\noindent
{Bozhidar Velichkov:\\
Laboratoire Jean Kuntzmann (LJK), 
Universit\'e Grenoble Alpes\\
B\^atiment IMAG, 700 Avenue Centrale, 
38401 Saint-Martin-d'H\`eres - FRANCE}\\
{\tt bozhidar.velichkov@univ-grenoble-alpes.fr}\\
{\tt http://www.velichkov.it}

\end{document}